\documentclass[11pt]{amsart} \textwidth=14.5cm \oddsidemargin=1cm
\usepackage[rgb]{xcolor}
\usepackage{tikz}
\usepackage{verbatim}
\usepackage{amsmath}
\usepackage{amsxtra}
\usepackage{amscd}
\usepackage{amsthm}
\usepackage{amsfonts}
\usepackage{amssymb}
\usepackage{eucal}
\usetikzlibrary{arrows}

\usepackage{latexsym,todonotes}

\usepackage[linktocpage=true]{hyperref}
\usepackage[all, cmtip]{xy}
\usepackage{stmaryrd}
\usepackage{color} 
\newtheorem{thm1}{Theorem}
\newtheorem{thm}{Theorem} [section]
\newtheorem{cor}[thm]{Corollary}
\newtheorem{lem}[thm]{Lemma}
\newtheorem{prop}[thm]{Proposition}

\theoremstyle{definition}

\theoremstyle{remark}

\numberwithin{equation}{section}

\begin{document}

%Referring commands:
\newcommand{\thmref}[1]{Theorem~\ref{#1}}
\newcommand{\secref}[1]{Section~\ref{#1}}
\newcommand{\lemref}[1]{Lemma~\ref{#1}}
\newcommand{\propref}[1]{Proposition~\ref{#1}}
\newcommand{\corref}[1]{Corollary~\ref{#1}}
\newcommand{\remref}[1]{Remark~\ref{#1}}
\newcommand{\eqnref}[1]{(\ref{#1})}

\newcommand{\exref}[1]{Example~\ref{#1}}

%Simplified symbols:
\newcommand{\nc}{\newcommand}
 \nc{\Z}{{\mathbb Z}}
 \nc{\C}{{\mathbb C}}
 \nc{\N}{{\mathbb N}}
 \nc{\F}{{\mf F}}
 \nc{\Q}{\mathbb{Q}}
 \nc{\la}{\lambda}
 \nc{\ep}{\epsilon}
 \nc{\h}{\mathfrak h}
 \nc{\n}{\mf n}
 \nc{\G}{{\mathfrak g}}
 \nc{\DG}{\widetilde{\mathfrak g}}
 \nc{\SG}{\breve{\mathfrak g}}
 \nc{\is}{{\mathbf i}}
 \nc{\V}{\mf V}
 \nc{\bi}{\bibitem}
 \nc{\E}{\mc E}
 \nc{\ba}{\tilde{\pa}}
 \nc{\half}{\frac{1}{2}}
 \nc{\hgt}{\text{ht}}
 \nc{\mc}{\mathcal}
 \nc{\mf}{\mathfrak} 
 \nc{\hf}{\frac{1}{2}}
\nc{\ov}{\overline}
\nc{\ul}{\underline}
\nc{\I}{\mathbb{I}}
\nc{\aaa}{{\mf A}}
\nc{\xx}{{\mf x}}
\nc{\id}{\text{id}}
\nc{\one}{\bold{1}}
\nc{\Qq}{\Q(q)}
\nc{\mA}{\mathcal{A}}
\nc{\mH}{\mathcal{H}}
\nc{\bk}{\mathbb{K}}

%color
\newcommand{\blue}[1]{{\color{blue}#1}}
\newcommand{\red}[1]{{\color{red}#1}}
\newcommand{\green}[1]{{\color{green}#1}}
\newcommand{\white}[1]{{\color{white}#1}}

%todo
\newcommand{\htodo}{\todo[inline,color=orange!20, caption={}]}

\title[A categorification of Hecke algebras with parameters $1$ and $v$]
{A categorification of Hecke algebras with parameters $1$ and $v$}
 
 \author[Huanchen Bao]{Huanchen Bao}
\address{Department of Mathematics, University of Maryland, College Park, MD 20742}
\email{huanchen@math.umd.edu}

\begin{abstract}
We categorify the Hecke algebra with parameters $1$ and $v$ using a variation of the category of Soergel bimodules.

\end{abstract}

%\vspace{.3cm}
\maketitle

\let\thefootnote\relax\footnotetext{{\em 2010 Mathematics Subject Classification.} Primary 17B10.}

%\setcounter{tocdepth}{1}
%\tableofcontents

%%%%%%%%%%
%%%%%%%%%%
%%%%%%%%%%
%%%%%%%%%%
%%%%%%%%%%%%%%%%%%%%%%%%%%%%%%%%%%%%%%%%%%%%%%%%%%%%%%%%%%%%%%%%%%%%%%55%%%%

%%%%%%%%%%%%%%%%%%%%%%%%%%%%%%%%%%%%%%%%%%%%%%%%%%%%%%%%%%%%%%%%%%%%%%55%%%%

\section*{Introduction}

\subsection{} 
Let $(W, S)$ be a Coxeter system, where $S$ denotes the set of simple reflections. Let $m_{s_1,s_2} \in \N \cup \{\infty\}$, such that $(s_1 s_2)^{m_{s_1,s_2}} =1 $ for $s_1, s_2 \in S$. We denote by $\ell(\cdot)$ the length function on $W$. Let $T = \bigcup_{w \in W} w S w^{-1}$ be the set of reflections. We denote by $\le $ the Bruhat order on $W$. 
We denote by $V$ (over $\mathbb{R}$) the geometric representation of $W$ and denote the root system by $\Phi = \Phi^+ \cup \Phi^-$ of $W$ in the sense of \cite[Section~5.4]{Hum90}. 
Let  $\{\alpha_s \vert s \in S \}$ be the collection of simple roots. Let $n (w) = \text{Card}( \Phi^+ \cap w(\Phi^-))$. We know $n(w) = \ell(w)$ for any $w \in W$. 

A weight function $L : W \rightarrow \Z$ is a function on $W$ such that $L(w_1w_2)= L(w_1) + L(w_2)$ whenever $\ell(w_1 w_2) = \ell(w_1) + \ell(w_2)$ for $w_1, w_2 \in W$. It follows that $L(s_1) = L(s_2)$ for any $s_1, s_2 \in S$ such that $m_{s_1,s_2}$ is odd.

In this note, we assume that a weight function $L$ is fixed such that 
\[
L(s)=0 \text{ or } 1, \quad \text{ for } s \in S.
\]
We define $S_{e} = S \cap L^{-1}(e)$ for $e \in \{0,1\}$. The only interesting case in this paper is when $L$ is not constant. 
\subsection{}
Let $\mA = \Z[v,v^{-1}]$ with a generic parameter $v$. For $s \in S$, we set $v_s= v^{L(s)} \in \mA$.  Let $\mathcal{H} = \mH_W$ be the $\mA$-algebra generated by $T_s (s \in S)$ subject to the relations:
\begin{align*}
(T_s - v^{-1}_s)(T_s+ v_s) &= 0, \quad \text{ for }s \in S\\
T_s T_{s'} T_s \cdots &= T_{s'} T_s T_{s'}\cdots, 
\end{align*}
where both products in the second relation have $m_{s,s'}$ factors for any $s \neq s' \in S$ such that $m_{s,s'}\neq \infty$. We write 
\begin{equation}\label{eq:Tw}
T_w = T_{s_1} \cdots T_{s_n}, \quad \text{ for any reduced expression }w =s_1 \cdots s_n \text{ with }s_i \in S.
\end{equation}
The set $\{T_w \vert w \in W\}$ forms an $\mA$-basis of $\mH$.
Let $\bar{\,} : \mH \rightarrow \mH$ be the $\mA$-semilinear bar involution such that $\overline{T_s}= T_s^{-1}$ and $\overline{v} =v^{-1}$. Note that for any $s \in S_0$, we have $T_s^2 =1$ and $\overline{T_s} = T_s$.

Thanks to \cite[Chap~5]{Lu03}, for any $w \in W$, there is a unique element $c_w$ such that
\begin{enumerate}
%	\item	The set $\{c_w \vert w \in W\}$ forms an $\mA$-basis of $\mH$, called the canonical basis;
	\item	$\overline{c_w} = c_w$;
	\item	$c_w = \sum_{y \in W} p_{y,w} T_y$ where
		\begin{itemize}	
			\item	$p_{y,w} = 0$ unless $y \le w$;
			\item	$p_{w,w} =1$;
			\item	$p_{y,w} \in v \Z[v]$ if $y < w$.
		\end{itemize}
\end{enumerate}
The set $\{c_w \vert w \in W\}$ forms an $\mA$-basis of $\mH$, called the canonical (or Kazhdan-Lusztig) basis.

\subsection{}
Soergel (\cite{Soe07}) categorified the Hecke algebras with equal parameters (that is, $L(s) =1$ for all $s \in S$) in terms of the category of Soergel bimodules. 

Following \cite{EW14}, we fix a Soergel realization $(\mathfrak{h}, \{\alpha_s\}, \{\alpha^\vee_{s}\})$ of $(W,S)$ over $\mathbb{R}$. This realization is faithful and Soergel's techniques can be applied. Let $R = \bigoplus_{m \ge 0} S^m(\mathfrak{h}^*)$, which we view as a graded $\mathbb{R}$-algebra with deg$(\mathfrak{h}^*) =2$. For $s \in S$, we define the graded $R$-bimodule $B_s = R \otimes _{R^s} R [1]$, where $[1]$ denotes the degree shifting. For any $w \in W$, we denote the standard bimodule associated with $w$ by $R_w$. Recall $R_w$ is isomorphic to $R$ as $\mathbb{R}$-modules and the $R$-bimodule structure is defined as : $f \cdot a = a \cdot w(f)$ for $f\in R$ and $a \in R_{w}$. 
 
 Let $\mathbb{BS}$Bim denote the full monoidal subcategory of $R$-Bim whose objects are Bott-Samelson bimodules. Let $\mathbb{S}$Bim denote the Karoubi envelope of $\mathbb{BS}$Bim, which is nowaday called the category of Soergel bimodules. Following \cite{Soe07, EW14}, we know $\mathbb{S}$Bim categorifies the Hecke algebra $\mc{H}$ of equal parameters. We have an algebra homomorphism from the Grothendieck group [$\mathbb{S}$Bim] to the Hecke algebra $\mc{H}$, where the images of the indecomposable objects are the canonical basis elements up to degree shift.

\subsection{}\label{subset:def}

 Now let $\mathbb{BS}$Bim$^L$ be the full monoidal subcategory of $R$-Bim generated by $R_{s}$ ($s \in S_0$) and $B_{s'}$ ($s' \in S_1$). For any expression $\underline{w}= s_{i_1} \cdots  s_{i_n} \in W$, we define the bimodule $B^L_{\underline{w}}$ as a product of $R_{s}$ ($s \in S_0$) and $B_{s'}$ ($s' \in S_1$) following the prescribed expression $\underline{w}$. It is easy to see that any objects in this category admits both standard and costandard filtrations (c.f. \cite{Soe07}). Hence we can define the character of any object in $\mathbb{BS}$Bim$^L$, denoted by $ch$. (Note that the definition of the character map makes sense as long as we have the standard basis $\{T_w \in \mc{H}\vert w \in W\}$ despite different multiplicative structure of $\mc{H}$ versus \cite{Soe07}.)

We denote by $\mathbb{S}$Bim$^L$ the Karoubi envelope of $\mathbb{BS}$Bim$^L$. We prove the following theorem in this note (which follows from Proposition~\ref{prop:1} and Proposition~\ref{prop:2}).

\begin{thm1}
For any $w \in W$, there exists a unique indecomposable bimodule (up to isomorphism) $B^L_w$ which occurs as a summand of $B^L_{\underline{w}}$ for any reduced expression $\underline{w}$ of $w$ such that $R_w[\nu]$ occurs in its standard filtration. The set $\{B^L_w [\nu] \vert w \in W, \nu \in \Z\}$ gives a complete list of indecomposable bimodules in $\mathbb{S}\text{Bim}^L$.

There is a unique isomorphism of $\mA$-algebras 
\begin{align*}
	\varepsilon: \mH & \longrightarrow [\mathbb{S}\text{Bim}^L],\\
					c_w & \mapsto B^L_w.
\end{align*}
The inverse of $\varepsilon$ is given by the character map $ch : [\mathbb{S}\text{Bim}^L] \rightarrow \mH$. 
%\begin{align*}
%ch : [\mathbb{S}\text{Bim}^L] &\longrightarrow \mH,\\
%B_s & \mapsto T_s, \text{ for } s \in S_1, \\
%R_{s'} &\mapsto T_{s'}, \text{ for } s \in S_0.
%\end{align*}

\end{thm1}

\subsection{}In the paper \cite{GT17}, Gobet and Thiel studied the generalized category of Soergel bimodules, with focus on type $A_2$. The category $\mathbb{S}\text{Bim}^L$ we constructed here is a subcategory of their category $\mc{C}$.

{\bf Acknowledgement:} The author would like to thank Xuhua He for helpful comments.

%%%%%%%%%%%%%%%%%%%%%%%%%%%%%%%%%%%%%%%%%%%%%%%%%%%%%%%%%%%%%%%%%%%%%%55%%%%

%%%%%%%%%%%%%%%%%%%%%%%%%%%%%%%%%%%%%%%%%%%%%%%%%%%%%%%%%%%%%%%%%%%%%%55%%%%
\section{Proof of the theorem}
\subsection{Coxeter groups}
In this section we review basics of Coxeter groups and their reflection subgroups. We refer to \cite{Hum90} for more details. Let $e \in \{0,1\}$ and $w \in W$. We define
 \begin{align*}
 &S_{e} = S \cap L^{-1}(e), \quad T_e = \bigcup_{w \in W} w S_e w^{-1},
\quad  \Phi_e = \{w(\alpha_s) \in V \vert w \in W, s \in S_e\},\\
 &\Phi_e^{\pm} = \Phi_e \cap \Phi^{\pm}, \quad n_e (w) =  \text{Card} \{\alpha \in  \Phi_e^+ \cap w(\Phi_e^-) \}. 
 \end{align*}
 Let $W_{S_0}$ be the parabolic subgroup of $W$ generated by $s \in S_0$. We have $T_0 \cap T_1 = S_0 \cap S_1 = \emptyset$. We are interested in the reflection subgroup $W' =\langle T_1 \rangle$ of $W$.

\begin{prop}\cite{Dy90, De89}\label{prop:Dyer}
The subgroup $W'$ of $W$ is itself a Coxeter group with simple reflections $S' = \{w \in T_1 \vert n_1(w) =1\}$. Moreover the restriction of $n_1$ on $W'$ coincides with the (new) length function of $(W', S')$.
\end{prop}

It is clear that the subspace spanned by $\Phi_1$ equipped with the natural $W'$-action coincides with the geometric representation of $W'$. We shall generally use $n_1$ to denote the (new) length function on $W'$ and reserve $\ell(\cdot)$ (or $n(\cdot)$) for the length function on $W$.
Let $m_{r,r'} \in \N \cup \{\infty \}$ such that $(r r')^{m_{r,r'}} =1$ for $r , r' \in S'$. %We give a description of the reduced expressions of elements in $S'$.

We first prove the Bruhat order on $W'$ (as a Coxeter group itself) is compatible with the Bruhat order on $W$. It follows from \cite[Theorem~3.3]{Dy90} that the set of reflections (with respect to the Coxeter system $(W', S')$) in $W'$ is exactly $T_1$. One can also see this fact from Corollary~\ref{cor:s}.

\begin{lem}\label{lem:Bruhat}
Let $w' \in W'$ such that $n_1(w' s_{\alpha}) > n_1 (w')$ for some $s_{\alpha} \in T_1$ $(\alpha \in \Phi^+)$. Then
\begin{enumerate}
	\item 		we  have $\ell(w' s_{\alpha}) > \ell (w')$; 
	\item		for any $g, g' \in W_{S_0}$, we have $ \ell (g w' s_{\alpha}) > \ell(gw')$ and $ \ell (g w' s_{\alpha}g') > \ell(gw'g')$.
\end{enumerate}
\end{lem}

\begin{proof}
Thanks to Proposition~\ref{prop:Dyer}, we know $n_1(\cdot)$ coincides with the length function on $W'$. Then thanks to \cite[Proposition~5.7]{Hum90}, we see that $w' (\alpha) \in \Phi^+_1 \subset \Phi^+$. Therefore we also have $\ell(w ' s_{\alpha} )  > \ell(w')$ by \cite[Proposition~5.7]{Hum90}. The first claim follows.

Now since $w' (\alpha) \in \Phi_1^+$ and $g \in W_{S_0}$, we must also have $gw' (\alpha) >0$, hence  $ \ell (g w' s_{\alpha}) > \ell(gw')$. 
On the other hand, we equivalently have $gw'g'\Big( g'^{-1}(\alpha) \Big) >0$, which means $ \ell (g w' g' g'^{-1} s_{\alpha}g') = \ell(g w'   s_{\alpha}g' ) > \ell(gw'g')$. The second claim follows.
\end{proof}

We then give a description of the set $S'$.
\begin{prop}\label{prop:S'}
Let $r \in T_1$. Then $r \in S\rq{}$ if and only if $r = w s w^{-1}$ for some $s \in S_1$ and $w \in W_{S_0}$. 

Moreover, the generator $r\in S\rq{}$ has a reduced expression (as an element in $W$) of the form $s_1 s_2 \cdots s_{k-1} s_k s_{k-1} \cdots s_2 s_1$, for $s_k \in S_1$ and $s_1, s_2, \dots, s_{k-1} \in S_0$.
%Let $r \in S'$. The element $r$ has a reduced expression (as an element in $W$) of the form $s_1 s_2 \cdots s_{k-1} s_k s_{k-1} \cdots s_2 s_1$, for $s_k \in S_1$ and $s_1, s_2, \dots, s_{k-1} \in S_0$.
\end{prop}

\begin{proof}
Following  Lemma~\ref{prop:S'}, we know that $n_1 (w s w^{-1}) = n_1(s) = 1$. So we have $w s w^{-1} \in S\rq{}$. This finishes the \lq\lq{}if\rq\rq{} direction. 

For the other direction, it suffices to prove the second statement.  Let $s_\alpha =r \in S'$ be a reflection of $V$ sending the root $\alpha$ $(\in \Phi^+) $ to $-\alpha$ with reduced expression $s_\alpha = s_1 \cdots s_n$ for $s_i \in S$. We know $  \Phi^+ \cap s_{\alpha}(\Phi^-)  = \{ \alpha_{s_n}, s_{n}(\alpha_{s_{n-1}}), \dots,\}$. By the definition of $S'$, we have $\Phi_1^+ \cap s_{\alpha}(\Phi^-) = \{\alpha\}$. Assume  $s_{n} s_{n-1} \cdots s_{k+1} (\alpha_{s_k}) = \alpha$ for $s_k \in S_1$. It also follows that $s_{n}, s_{n-1}, \dots, s_{k+1}, s_{k-1}, \dots, s_1 \in S_0 $.

Now we can write $s_{\alpha} = s_{n} s_{n-1} \cdots s_{k+1} s_{k} s_{k+1} \cdots s_{n-1} s_{n} = s_1 \cdots s_n $. We obtain that $s_{n} s_{n-1} \cdots s_{k+1}  = s_1 \cdots s_{k-1}$. Since we have reduced expressions on both sides, we must have $n =2k+1$ and $r= s_{\alpha} = s_1 s_2 \cdots s_{k-1} s_k s_{k-1} \cdots s_2 s_1$ being a reduced expression. 
\end{proof}

\begin{cor}\label{cor:s}
Let $w \in W_{S_0}$. The conjugation action of $s$ preserves the set $S'$.
\end{cor}

%%%%%%%%%%%%%%%%%%%%%%%%%%%%%%%%%%%%%%%%%%%%%%%%%%%%%%%%%%%%%%%%%%%%%%55%%%%

%%%%%%%%%%%%%%%%%%%%%%%%%%%%%%%%%%%%%%%%%%%%%%%%%%%%%%%%%%%%%%%%%%%%%%55%%%%

\subsection{Hecke algebras}
We denote by $\mH' = \mH_{W'}$ the Hecke algebra associated with the Coxeter subgroup $W'$ of $W$ with  generators $T_{r}'$ ($r \in S'$) subject to the relations (we write $v_r = v^{L(r)}$ for $r \in S'$):
\begin{align*}
(T'_{r} - v^{-1}_r)(T'_r+ v_r) &= 0, \quad \text{ for }r \in S'\\
T'_r T'_{r'} T'_r \cdots &= T'_{r'} T'_r T'_{r'}\cdots, 
\end{align*}
where both products in the second relation have $m_{r,r'}$ factors for any $r \neq r' \in S'$ such that $m_{r,r'}\neq \infty$. Note that we have $v_r =v_{r'}$ for any $r, r' \in S'$, thanks to the definition of the function $L$ and Proposition~\ref{prop:S'}. So this is a Hecke algebra with the weight function $L(r) =1$ for all $r \in S'$. 
%It follows from Corollary~\ref{cor:s} that conjugating by $s \in S_0$ defines an automorphism of $\mc{H}'$.

We write the canonical basis element in $\mH'$ as $c'_{w}$ for any $w \in W'$ to distinguish it from the canonical basis element $c_w$ in $\mH$ (since $W' \subset W$). But we shall see very soon they actually coincide.

\begin{thm}\label{thm:rho}
We have the $\mA$-algebra embedding $\rho: \mH' \rightarrow \mH$ such that 
\[
	\rho (T'_{r}) = T_r, \quad \text{ for } r \in S',
\]
where $T_r$ is the element in $\mH$ defined in \eqref{eq:Tw}.
Moreover, we have $\rho(T_w') = T_w$ and $\rho(c'_w) = c_w$ for any $ w \in W' \subset W$.
\end{thm}

\begin{proof}
Let $r \in S'$. Thanks to Proposition~\ref{prop:S'}, we have a reduced expression 
\begin{equation}\label{eq:r}
r= s_1 s_2 \cdots s_{k-1} s_k s_{k-1} \cdots s_2 s_1,
\end{equation}
for $s_k \in S_1$ and $s_1, s_2, \dots, s_{k-1} \in S_0$.  Therefore we have $L(r) = L(s_k)$, hence $v_r= v_{s_k}$.

We write $g = s_1 s_2 \cdots s_{k-1}$. Note that $T_{g^{-1}} = T^{-1}_g$. We first check the quadratic relations. We have
\[
\rho \Big( (T'_{r} - v^{-1}_{r})(T'_{r} + v_{r}) \Big) = T_{g} (T_{s_k} - v^{-1}_{s_k})(T_{s_k} + v_{s_k})T_{g}^{-1} =0.
\]

We then check the braid relations. Let $w = rr'r\cdots = r'r r' \cdots$ with $m_{r,r'}$ factors for any $r \neq r' \in S'$ such that $m_{r,r'}\neq \infty$. We first replace $r, r' \in S'$ by the reduced expressions obtained in Proposition~\ref{prop:S'}.

After the replacement, if we obtain reduced expressions in $W$, then we naturally have 
\[
T_w = T_r T_{r'} T_r \cdots = T_{r'} T_r T_{r'}\cdots, \text{ as an identity in }\mc{H}.
\]

If the expression we obtained for $w$ is not reduced in $W$, then we can apply the braid relations and quadratic relations to obtain a reduced expression of $w$ from the products $rr'r\cdots = r'r r' \cdots$ (c.f.\cite{Hum90}). Since $n_1(w) = m_{r,r'}$,  the reduced expression of $w$ (in $W$) contains exactly $m_{r,r'}$ factors in $S_1$.
Hence we can only apply the quadratic relations $s^2=1$ with $s \in S_0$ (which reduced the number of factors), but never the relations $s^2 =1 $ for $s \in S_1$. Then since we also have $T_{s}^2 =1 $ in $\mH$ for $s \in S_0$. We are able to perform those operations in the Hecke algebra $\mH$ as well. Therefore we shall still have 
\[
T_w = T_r T_{r'} T_r \cdots = T_{r'} T_r T_{r'}\cdots.
\]
This shows that $\rho$ is an algebra homomorphism. Entirely similar argument shows that we also have $\rho (T_w') = T_w$.

Following the reduced expression of $r$ in \eqref{eq:r}, it is easy to see we have $\rho(\overline{T'_r}) = \overline{T_r}$, that is $\rho$ commutes with the bar involutions.  Now in order to prove that $\rho(c'_w) = c_w$, it suffices to prove the Bruhat order in $W'$ is compatible with the Bruhat order in $W$, which follows from Lemma~\ref{lem:Bruhat}.
\end{proof}

So from now on we shall omit the superscripts $'$ for the generators $T'_r \in \mH'$ and for the elements $c'_w \in \mH'$.

Now combining Theorem~\ref{thm:rho} with Lemma~\ref{lem:Bruhat}, we obtain the following corollary of the canonical basis element $c_w \in \mH$ for arbitrary $w \in W$, which is also proved in \cite[\S6.1]{Lu03}.
\begin{cor}\label{cor:scw}
For any $w \in W$ and $g \in W_{S_0}$, we have $T_gc_{w} = c_{gw}$ and $c_{w} T_g = c_{wg}$. 
%So now any $c_{w}$ and be written as $T_{g} c_{w'}$ for some $w' \in W'$ and $g \in W_{S_0}$.
\end{cor}

%%%%%%%%%%%%%%%%%%%%%%%%%%%%%%%%%%%%%%%%%%%%%%%%%%%%%%%%%%%%%%%%%%%%%%55%%%%

%%%%%%%%%%%%%%%%%%%%%%%%%%%%%%%%%%%%%%%%%%%%%%%%%%%%%%%%%%%%%%%%%%%%%%55%%%%

\subsection{Soergel bimodules}
Let $s \in S$. For any reflection $r = w s w^{-1}$ in $W$, we define $\alpha_r = w(\alpha_s) \in \mathfrak{h}$ and $\alpha^\vee_r = w(\alpha^\vee_s) \in \mathfrak{h}^*$. We assume that $(\mathfrak{h}, \{\alpha_r\}, \{\alpha^\vee_r\})$ for $r \in S'$ is also a realization of the Coxeter system $(W', S')$. Certainly if  $\mathfrak{h}$ is a reflection faithful representation of $W$ as in \cite{Soe07}, then $\mathfrak{h}$ is also a reflection faithful representation of $W'$.

\begin{lem}\label{lem:iso}
Let $s \in S$ and $w \in W$. We have $R_{w} \otimes_R B_{s} \otimes_R R_{w^{-1}} \cong R \otimes_{R^{wsw^{-1}}}R$ as $R$-bimodules.
\end{lem}

\begin{proof}
Let $r = w s w^{-1}$ be a reflection in $W$. We define $\partial_{r}  = w \partial_{s} w^{-1}: R \rightarrow R^{r}(-2)$ such that 
\[
\partial_{r}(f) = \frac{f - r(f)}{\alpha_{r}}.
\]
%Then similar to \cite[Claim~3.9]{EW13}, 

We obtain that $R \otimes_{R^{r}}R$ is a free left (or right) $R$-module with basis
\[
c^r_1 = 1 \otimes 1 \in R  \otimes_{R^{r}} R , \qquad c^r_{r} = \Delta_{s}  \in R  \otimes_{R^{r}} R ,
\]
such that 
\[
x c^r_{r} = c^r_{r} x, \qquad x c^r_1 =  c^r_1 s(x) + \partial_{r}(x) c^r_{r}, \quad x \in R.
\]

On the other hand, we know that $B_{s'}$ is a free left (or right) $R$-module with basis 
\[
c^{s}_1 = 1 \otimes 1 \in R  \otimes_{R^{s}} R , \qquad c^{s}_{s} = \Delta_{s} \in R  \otimes_{R^{s}} R,
\]
such that 
\[
r c^{s}_{s} = c^{s}_{s} r, \qquad r c^{s}_1 =  c^{s}_1 s(r) + \partial_{s}(r) c^{s}_{s}.
\]
Hence by direct computation we see $R_{w} \otimes_R B_{s} \otimes_R R_{w^{-1}} \cong R \otimes_{R^{r}}R$ as $R$-bimodules.
\end{proof}

%
%\begin{lem}\label{lem:iso}
%Let $s,s' \in S$. We have $R_{s} \otimes_R B_{s'} \otimes_R R_s \cong R \otimes_{R^{ss's}}R$ as $R$-bimodules.
%\end{lem}
%
%
%
%\begin{proof}
%Let $r = ss's$ be a reflection in $W$. We define $\partial_{r}  = s \partial_{s'} s: R \rightarrow R^{r}(-2)$ such that 
%\[
%\partial_{r}(f) = \frac{f - r(f)}{\alpha_{r}}.
%\]
%%Then similar to \cite[Claim~3.9]{EW13}, 
%
%We obtain that $R \otimes_{R^{ss's}}R$ is a free left (or right) $R$-module with basis
%\[
%c^r_1 = 1 \otimes 1 \in R  \otimes_{R^{r}} R , \qquad c^r_{r} = \Delta_{s'}  \in R  \otimes_{R^{r}} R ,
%\]
%such that 
%\[
%x c^r_{r} = c^r_{r} x, \qquad x c^r_1 =  c^r_1 s'(x) + \partial_{r}(x) c^r_{r}, \quad x \in R.
%\]
%
%
%
%
%
%On the other hand, we know that $B_{s'}$ is a free left (or right) $R$-module with basis 
%\[
%c^{s'}_1 = 1 \otimes 1 \in R  \otimes_{R^{s'}} R , \qquad c^{s'}_{s'} = \Delta_{s'} \in R  \otimes_{R^{s'}} R,
%\]
%such that 
%\[
%r c^{s'}_{s'} = c^{s'}_{s'} r, \qquad r c^{s'}_1 =  c^{s'}_1 s'(r) + \partial_{s'}(r) c^{s'}_{s'}.
%\]
%Hence by direct computation we see $R_{s} \otimes_R B_{s'} \otimes_R R_s \cong R \otimes_{R^{ss's}}R$ as $R-$bimodules.
%\end{proof}

Now for any $r \in S'$ such that $r = wsw^{-1}$ with $w \in W_{S_0}$ and $s \in S_1$, we define the $R$-bimodule (independent of the reduced expression of $r$)
\[
B'_{r} =  R \otimes_{R^r} R \cong R_w B_{s} R_{w^{-1}}.
\]
Recall the definitions of $\mathbb{BS}$Bim$^L$ and $\mathbb{S}$Bim$^L$ in \S\ref{subset:def}.  Let $\mathbb{S}\text{Bim}_{W'}$ be the Karoubi envelop of the monoidal subcategory of $\mathbb{S}\text{Bim}^L$ generated by $B'_r$ for $r \in S'$. It is easy to see that $\mathbb{S}$Bim$_{W'}$ the category of Soergel Bimodules associated with the realization $(\mathfrak{h}, \{\alpha_r\}, \{\alpha^\vee_r\})$ of the Coxeter system $(W', S')$.

\begin{prop}\label{prop:W'}
There is a unique isomorphism of $\mA$-algebras 
\begin{align*}
	\varepsilon: \mH' & \longrightarrow [\mathbb{S}\text{Bim}_{W'}],\\
					c_r & \longmapsto B'_r, \quad \text{ for } r \in S'.
\end{align*}
The inverse of $\varepsilon$ is given by the character map $ch : [\mathbb{S}\text{Bim}_{W'}] \rightarrow \mH'$. 

For any $w \in W'$, there exists a unique indecomposable bimodule (up to isomorphism) $B_w\rq{}$ which occurs as a summand of $B_{\underline{w}}\rq{}$ for any reduced expression $\underline{w}$ of $w$ (in $W\rq{}$) such that $R_w[\nu]$ occurs in its standard filtration. The set $\{B_w\rq{} [\nu] \vert w \in W', \nu \in \Z\}$ gives a complete list of indecomposable bimodules in $\mathbb{S}\text{Bim}_{W'}$.
\end{prop}

\begin{proof}This is exactly Soergel's categorification theorem of the Hecke algebra $\mc{H}'$ with the realization $(\mathfrak{h}, \{\alpha_r\}, \{\alpha^\vee_r\})$ of the Coxeter system $(W', S')$.
\end{proof}

The following lemma is the categorical analog of Corollary~\ref{cor:s}.
\begin{lem}\label{lem:equiv}
Let $w \in W_{S_0}$. The functor $R_{w} \otimes - \otimes R_{w^{-1}} : \mathbb{S}$Bim$^L \rightarrow \mathbb{S}$Bim$^L$ is an equivalence of monoidal category. In particular, the functor restricts to an equivalence  $R_{w} \otimes - \otimes R_{w^{-1}} : \mathbb{S}$Bim$_{W'} \rightarrow \mathbb{S}$Bim$_{W'}$.
\end{lem}

\begin{proof}
The first statement is obvious. The second statement follows from the fact that the conjugation action of $w$ preserves the set $S'$. 
\end{proof}

\begin{lem}\label{lem:decom}
We have the following equivalence of additive categories
\[
 \mathbb{S}\text{Bim}^L \cong \bigoplus_{w \in W_{S_0}} R_{w} \otimes  \mathbb{S}\text{Bim}_{W'} \cong \bigoplus_{w \in W_{S_0}} \mathbb{S}\text{Bim}_{W'} \otimes R_w.
\] 
\end{lem}

\begin{proof}
Thanks to Lemma~\ref{lem:equiv}, any tensor product of $R_s (s \in S_0)$ and $B_{s'} (s' \in S_1)$ lies in $R_{w} \otimes  \mathbb{S}\text{Bim}_{W'}$ (or  $\mathbb{S}\text{Bim}_{W'} \otimes R_w$) for some $w \in W_{S_0}$, which is clearly closed under taking direct summands. 
\end{proof}

\begin{prop}\label{prop:1}
There is a unique isomorphism of $\mA$-algebras 
\begin{align*}
	\varepsilon: \mH & \longrightarrow [\mathbb{S}\text{Bim}^L],\\
					c_s & \longmapsto 
						\begin{cases}
							B_s, &\text{ if } s \in S_1;\\
							R_s, &\text{ if }s \in S_0.
						\end{cases} 
\end{align*}
The inverse of $\varepsilon$ is given by the character map $ch : [\mathbb{S}\text{Bim}^L] \rightarrow \mH$. 
%\begin{align*}
%ch : [\mathbb{S}\text{Bim}^L] &\longrightarrow \mH,\\
%B_s & \mapsto T_s, \text{ for } s \in S_1, \\
%R_{s'} &\mapsto T_{s'}, \text{ for } s \in S_0.
%\end{align*}

For any $w \in W$, there exists a unique indecomposable bimodule (up to isomorphism) $B_w^L$ which occurs as a summand of $B_{\underline{w}}^L$ for any reduced expression $\underline{w}$ of $w$ such that $R_w[\nu]$ occurs in its standard filtration. The set $\{B_w^L [\nu] \vert w \in W, \nu \in \Z\}$ gives a complete list of indecomposable bimodules in $\mathbb{S}\text{Bim}^L$.
\end{prop}

\begin{proof}The statement about indecomposable bimodules follows from Lemma~\ref{lem:decom} and Proposition~\ref{prop:W'}. Note that we have $B_s \cong B_s^L$ for $s \in S_1$. Let us prove $\varepsilon$ is an algebra homomorphism.

It suffices to consider the case when $W$ is a dihedral group with $S= \{s, r\}$. We can assume $L(s) =0$ and $L(r) =1$. The other two cases is either trivial or follows from \cite{Soe07}. The quadratic relations are obvious. We prove the braid relation 
\[T_{r} T_s T_r T_s \cdots = T_s T_r T_s T_r \cdots ,
\]
where both products have (necessarily even) ${m_{r,s}}$ factors.

Let us first consider the subalgebra  $\mc{H}'$ generated by $T_{srs} = T_s T_r T_s$ and $T_{r}$. Note that $L(r) = L(srs) =1$. Using the relation $T_{s}^2 =1$, we rewrite the braid relation as
\[
T_{srs} T_{r} T_{srs} T_{r} \cdots = T_{r} T_{srs} T_{r} T_{srs} \cdots,
\]
where both products have $m_{srs, r} = {m_{r,s}}/{2}$ factors. Thanks to Proposition~\ref{prop:W'}, we have the algebra homomorphism 
\begin{align*}
	\varepsilon: \mH' & \longrightarrow  [\mathbb{S}\text{Bim}_{W'}] \hookrightarrow [\mathbb{S}\text{Bim}^L],\\
					T_r+v & \longmapsto [B_r\rq{}] \cong [B_r] ;\\
					T_{srs} +v & \longmapsto [ R \otimes_{R^{srs}}R].
\end{align*}
Thanks to Lemma~\ref{lem:iso}, we have $R_{s} \otimes_R B_{r} \otimes_R R_s \cong R \otimes_{R^{srs}}R$. This finishes the proof that $\varepsilon$ is an algebra homomorphism. 
\end{proof}
The following corollary can be regarded as the categorical analog of Corollary~\ref{cor:scw}.
\begin{cor}\label{cor:sw}
For $g \in W_{S_0}$, we have $R_{g} \otimes B_w^L \cong B_{gw}^L$ and $B_w^L \otimes R_g \cong B_{wg}^L$ in $\mathbb{S}\text{Bim}^L$.
\end{cor}

\begin{proof}
The first isomorphism follows from the equivalence $R_g \otimes - : \mathbb{S}\text{Bim}^L \rightarrow \mathbb{S}\text{Bim}^L$ and the characterization of $B_{gw}$. The second one is similar.
\end{proof}
%
%\begin{rem}
%We remind the reader that here $B_w$ or $B_{sw}$ denotes the indecomposable objects in $\mathbb{S}\text{Bim}^L$. The Corollary~\ref{cor:sw} clearly do not hold in the usual category of Soergel bimodules. 
%\end{rem}

Recall we have identified $\mathbb{S}\text{Bim}_{W'}$ with a full subcategory of $\mathbb{S}\text{Bim}^L$, thanks to the isomorphism of bimodules in Lemma~\ref{lem:iso}. It is straightforward that we have $B_w\rq{}  \cong B_w^L \in \mathbb{S}\text{Bim}^L$ for $w \in W\rq{}$.

%
%
%\begin{prop}
%We have the following fully faithful functor 
%\begin{align*}
%{\wp}: \mathbb{S}\text{Bim}_{W'} &\longrightarrow \mathbb{S}\text{Bim}^L\\
%B'_r &\mapsto R_{s_1} R_{s_2} \cdots R_{s_{k-1}} B_{s_k} R_{s_{k-1}} \cdots R_{s_2} R_{s_1},
%\end{align*}
%for $r \in R$ with reduced expression as in Proposition~\ref{prop:S'}.
%\end{prop}
%\begin{rem}
%The $R$-bimodule on the right hand side does not depends on the choice of the reduced expression of $r \in R$ by a straightforward generalization of Lemma~\ref{lem:iso}.
%\end{rem}

\begin{cor}\label{cor:cdrho}
We have the following commutative diagrams of $\mA$-algebras
\[
\xymatrix{[\mathbb{S}\text{Bim}_{W'}] \ar[r] \ar[d]^{ch} & [\mathbb{S}\text{Bim}^L] \ar[d]^{ch}\\
\mH' \ar[r]^{\rho} & \mH
}
\]

\end{cor}

%\subsection{Kazhdan-Lusztig basis}
We now assume that $(\mathfrak{h}, \{\alpha_s\}, \{\alpha^\vee_{s}\})$ ($s\in S$) is a realization of $(W,S)$ over $\mathbb{R}$, such that $(\mathfrak{h}, \{\alpha_r\}, \{\alpha^\vee_{r}\})$ ($r \in S'$) is also  realization of $(W',S')$. In this setting we can apply results from \cite{EW14}.

\begin{prop}\label{prop:2}
The map $\text{ch} : [\mathbb{S}\text{Bim}^L] \rightarrow \mH$ sends $[B_w^L]$ to $c_w$ for any $w \in W$.
\end{prop}

\begin{proof}
Thanks to \cite{EW14}, we know the character map $ch: \mathbb{S}\text{Bim}_{W'} \rightarrow \mH'$ sends $B_{w}^L$ to $c_{w}$ for $w \in W'$. Then thanks to Theorem~\ref{thm:rho}, we know that $\rho \circ ch (B_{w}^L) = c_{w}$. Then by the commutative diagram in Corollary~\ref{cor:cdrho}, we see that $ch(B_{w}^L) = c_{w}$ for any $w \in W' \subset W$.

Hence thanks to the following commutative diagram ($s \in S_0$)
\[
\xymatrix{[\mathbb{S}\text{Bim}^L] \ar[r]^{R_{s} \otimes -} \ar[d]^{ch} & [\mathbb{S}\text{Bim}^L] \ar[d]^{ch}\\
\mH \ar[r]^{T_{s} \cdot} & \mH
},
\]
and Corollary~\ref{cor:scw}\&~\ref{cor:sw}, we have $ch(B_{w}^L) = c_w$ for any $w \in W$.
\end{proof}

%%%%%%%%%%%%%%%%%%%%%%%%%%%%%%%%%%%%%%%%%%%%%%%%%%%%%%%%%%%%%%%%%%%%%%55%%%%

%%%%%%%%%%%%%%%%%%%%%%%%%%%%%%%%%%%%%%%%%%%%%%%%%%%%%%%%%%%%%%%%%%%%%%55%%%%

\end{document}